\documentclass[12pt,oneside,a4paper,leqno]{article}
\usepackage[utf8]{inputenc}
\usepackage[english]{babel}
\usepackage[all]{xy}
\CompileMatrices

\newdir{ >}{!/10pt/@{ }*@{>}}
\newdir{< }{!/10pt/@{ }*@{<}}

\usepackage{graphicx}
\usepackage{amsfonts,amssymb}
\usepackage[centertags]{amsmath}

\usepackage{amsthm} %
\newcounter{theorems}

\theoremstyle{plain}

\swapnumbers
\newcounter{lemma}
\numberwithin{equation}{section}

\newtheoremstyle{par}%
     {\topsep}%
     {\topsep}%
     {\itshape}%
     {}%
     {\bfseries}%
     {}%
     {.5em}%
     {}%

\newtheoremstyle{parrm}%
     {\topsep}%
     {\topsep}%
     {\normalfont}%
     {}%
     {\itshape}%
     {}%
     {.5em}%
     {}%

\theoremstyle{plain}
\numberwithin{equation}{section}
\newtheorem{lemma}[equation]{Lemma}
\newtheorem{theo}[equation]{Theorem}

\theoremstyle{definition}

\theoremstyle{remark}
\newtheorem{remark}[equation]{Remark}

\theoremstyle{par}

\theoremstyle{parrm}

\makeatletter
\def\tagform@#1{\maketag@@@{\ignorespaces#1\unskip\@@italiccorr}}

\makeatother

\usepackage{paralist}
\usepackage{comment}
\usepackage{subfigure}
\newcommand{\RR}{\mathbb{R}}

\newcommand{\ZZ}{\mathbb{Z}}

\newcommand{\from}{\colon}

\usepackage{bm}
\newcommand{\simbolovettore}[1]{{\boldsymbol{#1}}}
\newcommand{\va}{\simbolovettore{a}}

\newcommand{\vn}{\simbolovettore{n}}

\newcommand{\vq}{\simbolovettore{q}}

\newcommand{\vv}{\simbolovettore{v}}
\newcommand{\vw}{\simbolovettore{w}}
\newcommand{\vx}{\simbolovettore{x}}

\newcommand{\zero}{\boldsymbol{0}}

\newcommand*\de{\mathop{}\!\mathrm{d}}

\usepackage[textwidth=28mm]{todonotes}
\presetkeys{todonotes}{size=\scriptsize,backgroundcolor=yellow!30,linecolor=red}{}

\usepackage{multirow,rotating}

\usepackage{url}
\newcommand{\mgrad}{\nabla_M}

\newcommand{\conf}[2]{\mathbb{F}_{#1}(#2)}

\newcommand{\eucnorm}[1]{\left\lvert{#1}\right\rvert}
\newcommand{\mnorm}[1]{{\left\lVert#1\right\rVert}_M}
\newcommand{\mscalar}[2]{{\left\langle{#1},{#2}\right\rangle}_{M}}
\newcommand{\action}{\mathcal{A}}
\newcommand{\dt}{\protect\de{t}}
\newcommand{\Sym}{\mathfrak{S}}
\newcommand{\Iso}{\operatorname{Iso}}
\newcommand{\TT}{\mathbb{T}}
\newcommand{\nop}{\nu}
\newcommand{\XX}{\mathcal{X}}
\newcommand{\dist}{\operatorname{dist}}
\newcommand{\codim}{\operatorname{codim}}

\begin{document}
\pagenumbering{arabic}

\title{%
Symmetries and periodic orbits for the $n$-body problem: about the computational approach
}

\author{D.L.~Ferrario}

\date{%
\today}
\maketitle

\begin{abstract}
The main problem is to understand and to find periodic symmetric orbits in the $n$-body
problem, in the sense of finding methods to prove or compute their
existence, and more importantly to describe their qualitative and
quantitative properties.  In order to do so, and in order to classify such
orbits and their symmetries, computers have been extensively used in many
ways since decades.  We will focus on some very special symmetric orbits,
which occur as symmetric critical points (local minimizers) of the
gravitational Lagrangean action functional.  The exploration of the loop
space of the $n$-point configuration space, raised some computational and
mathematical questions that couldd be interesting.  The aim of the article
is to explain how such questions and issues were %
considered in the development of a software package that combined symbolic
algebra, numerical and scientific libraries, human interaction and
visualization. %

\noindent {\em MSC Subject Class\/}: 
\vspace{0.5truecm}
70F10 %

\end{abstract}

\section{Introduction}

\textbf{Configuration spaces:}
Let $n\geq 2$ and $d\geq 1$ be integers. 
Let $\vn$ denote the set $\vn =\{1,2,\ldots, n\}$. 
Let $E=\RR^d$ be the $d$-dimensional euclidean space. 
Elements of $E^n$ are denoted by
$\vq=(\vq_1,\vq_2,\ldots, \vq_n)$,
where $\vq_j \in E$ for $j\in \vn$. 
The \emph{collision set} is
\[
\Delta = \bigcup_{i<j}\{ \vq \in E^n : \vq_i = \vq_j \},
\]
and 
the \emph{configuration space}  
of $n$ points in $E$ is 
\[
\conf{n}{E} = E^n \smallsetminus \Delta = \{ \vq \in E^n : \vq_i \neq \vq_j \}. 
\]

For $j\in \vn$,  let $m_j>0$ be called positive \emph{masses}. 
Assume that
\[
\sum_{j=1}^n m_j = 1.
\]
Let $\mscalar{*}{*}$ denote the \emph{mass-metric} on
(the tangent vectors of)
$E^n$, defined as
\[
\mscalar{\vv}{\vw} = \sum_{j=1}^n m_j \vv_j \cdot \vw_j,
\]
where $\vv_j \cdot \vw_j$ is the euclidean scalar product in (the tangent space of)
$E$. Let $\eucnorm{\vv_j}$ denote the euclidean norm  of a
vector $\vv_j$ in $E$.
The norm corresponding to the mass-metric is
$\mnorm{\vv} = \sqrt{\mscalar{\vv}{\vv}}$.
Let $\alpha>0$ be a fixed homogeneity parameter, and
$U \from \conf{n}{E} \to \RR$ the potential function defined as
\begin{equation}\label{eq:potential}
U(\vq) = \sum_{1\leq i<j \leq n} \dfrac{m_im_j}{\eucnorm{\vq_i - \vq_j}^\alpha}~.
\end{equation}
The gradient of $U$ with respect to the mass-metric is denoted by
$\mgrad U$, and in standard coordinates $\vq_j$ has $d$-dimensional $j$-th component
\[
\left( \mgrad U \right)_j = \dfrac{1}{m_j} \dfrac{\partial U}{\partial \vq_j}~.
\] 

\textbf{Newton and Lagrange equations:} 
Given the potential $U$, 
systems of interacting bodies can be modeled by their \emph{Newton equations}
\begin{equation}\label{eq:newton}
m_j \dfrac{d^2 \vq_j}{dt^2} = \dfrac{\partial U}{\partial q_j} (\vq), ~ j\in\vn, 
\iff
\dfrac{d^2\vq}{dt^2} = \mgrad U(\vq).
\end{equation}
The potential $U$ is smooth on the configuration space $\conf{n}{E}$,
and goes to infinity on the set $\Delta$ of collisions. 
This fact implies that solutions are generally smooth, but \emph{singularities}
can occur: $n$-body orbits can collide, or become unbounded in a finite time \cite{MR0495348,MR1166640}.
Furthermore, the general topology of the configuration space 
and $\Delta$ is held
responsible for the possibility of \emph{chaotic trajectories} \cite{Devaney81}.
For any open interval $\Omega \subset \RR$,
let $H^1(\Omega,E^n)$ be the Hilbert Sobolev space $W^{1,2}(\Omega,E^n)$ 
of all the functions  $\vq \in L^2(\Omega,E^n)$ with weak derivative in $L^2(\Omega,E^n)$. 
If $\Omega$ is a bounded interval, 
let $H^1_0(\Omega,E^n) \subset H^1(\Omega,E^n)$ be the 
space of all functions $\vq \in H^1(\Omega,E^n)$ vanishing
on the boundary of $\Omega$. 
The \emph{Lagrangian Action Functional} is defined on $H^1(\Omega,E^n)$ as
\[
\action(\vq) = \int_{\Omega} \frac{1}{2} \mnorm{\dot \vq}^2  + U(\vq) \dt.
\]
If $\vq \in H^1(\Omega, \conf{n}{E})$ is a collisionless solution of \eqref{eq:newton}, 
then it is a critical point for the action functional.  
On the other hand, the action functional $\action$ is finitely defined also on colliding 
trajectories (unless assuming the strong-force condition $\alpha\geq 2$),
and hence local minimizers \emph{a priori} need not be regular points. 

\textbf{Periodic and symmetric trajectories:}

Consider the symmetric group $\Sym_n$ of all permutations in the set 
of indices $\vn=\{1, 2, \ldots, n\}$, and 
the orthogonal group $O(d)$ of $E$. 
The symmetric group 
$\Sym_n$ 
acts on the left on $E^n$ by setting
\[
\sigma \cdot (\vq_1,\vq_2, \ldots, \vq_n) = (\vq_{\sigma^{-1}(1)}, \vq_{\sigma^{-1}(2)}, 
\ldots \vq_{\sigma^{-1}(n)}) 
\]
for each permutation $\sigma\in \Sym$. 
The orthogonal group 
$O(d)$ acts naturally on $E$, and hence it acts diagonally on $E^n$ by setting
\[
g  \cdot (\vq_1,\vq_2, \ldots, \vq_n) = (g \vq_1 , g \vq_2, 
\ldots g \vq_n) 
\]
for each element $g \in O(d)$. 
The two actions commute, and the direct product 
$\Sym_n  \times O(d)$ acts on $E^n$ accordingly.

Now, consider the space $C^0(\RR,E^n)$  of all continuous trajectories in $E^n$. 
The group $\Iso(\RR)$ of all isometries of the real line acts on $C^0(\RR,E^n)$ as
$(\tau \cdot \vq)(t) = \vq(\tau^{-1}(t))$ for each $\tau\in \Iso(\RR)$. 
Hence, $\Iso(\RR)\times \Sym_n \times O(d)$ acts on 
$C^0(\RR,E^n)$
as 
\[
\begin{aligned}
\left( (\tau, \sigma, g ) \cdot \vq \right) (t) 
& = 
((\sigma,g) \cdot \vq )(\tau^{-1}(t)) 
\\
& =
\left(g \vq_{\sigma^{-1}}(\tau^{-1}(t)),
g \vq_{\sigma^{-2}}(\tau^{-1}(t)),
\ldots,
g \vq_{\sigma^{-n}}(\tau^{-1}(t)
\right)
\\
\end{aligned}
\]
for each $\vq\from \RR \to E^n$.

A \emph{periodic trajectory} of period $T>0$  
is simply a trajectory $\vq \from \RR \to E^n$ such that 
$\tau_T \cdot \vq = \vq$, 
where $\tau_T \in \Iso(\RR)$ is the translation in time
$\tau_T(t) = t + T$. 
If with a slight abuse of notation 
the same symbol $T$ denotes the infinite cyclic subgroup in $\Iso(\RR)\times \Sym_n \times O(d)$ 
generated by $\tau_T$, then 
$\vq$ is periodic if and only if $\vq \in C^0(\RR,E^n)^T$,
where as usual 
when $G$ acts on a set $X$ the subspace  $X^G \subset X$ is defined as 
as $X^G = \{ x \in X : \forall g\in G, gx = x  \}$. 
If $\TT_T$ denotes the circle
$\RR/T\ZZ$, the space of all $T$-periodic trajectories
is endowed with the Hilbert space structure given by 
\[
X= H^1(\TT_T,E^n) = H^1([0,T],E^n) \cap C^0(\RR, E^n)^T  
\]
Furthermore, the full group $\Iso(\RR)\times \Sym_n \times O(d)$
acts on  $X$ by restriction. One can replace $\Iso(\RR)$ with $\Iso(\TT_T)$ 
for simplicity.

The action functional can be defined on  
$X$ as 
\[
\action(\vq) = \int_{0}^T \frac{1}{2} \mnorm{\dot \vq}^2  + U(\vq) \dt
\]
for all $\vq \in X$. 
Given a choice of masses $m_1,\ldots, m_n$, let 
$\Sym_n^m$  denote the subgroup of $\Sym_n^m$ 
consisting of all permutations  $\sigma$ 
in $\Sym_n$ such that $\forall j \in \vn, m_j = m_{\sigma(j)}$.  
If all masses are equal, then $\Sym_n^m = \Sym_n$.
If they are all different, then $\Sym_n^m = 1$. 
Now, observe that the functional $\action$ is invariant with respect 
to all elements in the product $\Iso(\RR)\times \Sym_n^m \times O(d)$.

The main procedure now is to consider Palais' \emph{Principle of symmetric criticality}:
if $G$ is any subgroup of $\Iso(\TT_T)\times \Sym_n^m \times O(d)$,
then
$\action^G$ denotes the restriction of $\action$ to the fixed subspace $X^G$,
and it happens that 
any critical point of 
$X^G$ is a critical point of $X$,
provided it is a non-colliding orbit.   
The main issues now are two: which symmetry groups allow to identify and to `understand' 
some classes of critical points (such as local symmetric minimizers),
and how to approximate such periodic orbits, if any. 

Now, let $G$ be a subgroup of 
$\Iso(\TT_T)\times \Sym_n^m \times O(d)$. The projections on the three
factors yield three group homomorphisms
$\tau\from G \to  \Iso(\TT_T)$,
$\sigma\from G \to \Sym_n^m$,
$\rho\from G \to O(d)$. 
The image of $\tau$ is a finite subgroup of order $l$ in $\Iso(\TT_T) \cong O(2)$. 
The group can  be just a single reflection in time (symmetry of \emph{brake type}),
i.e. a dihedral subgroup of $O(2)$ of order $2$, or
it can be 
a dihedral subgroup of $O(2)$ of order $l$, with $l>2$ 
(symmetry of \emph{dihedral type}),
or a cyclic group of rotations in $\TT_T$ (symmetry of \emph{cyclic type})
of order $l$ with $l>1$. The fundamental domain of the 
$G$-action on $\TT_T$ is therefore defined as $T$
divided by the order of the image of $\tau$. 
In all the three cases the \emph{fundamental domain} is the 
interval $[0,T/{l}]$.
Since the potential is homogeneous of degree $-\alpha$,
without loss of generality we can assume that $T = l$.
For dihedral symmetries let $\tau(h_0)$ and $\tau(h_1)$ the time-reflections
fixing $0$ and $1$ respectively in $\TT_T$. If the symmetry is of brake type,
then $h_0=h_1$, otherwise $h_0\neq h_1$ and $\tau(h_0),\tau(h_1)$ 
are generators of $\Im(\tau)$. 
For a cyclic type, let $r$ be the cyclic generator such that
$\tau(r)(0) = 1$ in $\TT_T$.  

Let $K=\ker \tau \subset G$. 
Then $\Im(\tau) = G/K$. 
The image of $\tau$ in $\Iso(\TT_T)$ is either 
generated by $\tau(r)$, or by $\tau(h_0)$ and $\tau(h_1)$ 
for suitable $r,h_0,h_1$ defined as above. 
Since 
$\vq \in X^G \iff 
\forall t \in \TT_T, 
\vq(\tau(g)(t)) = (\sigma\times \rho)(g)\vq(t)$, 
\[
X^G \cong  \{
\vq \from \TT_T \to (E^n)^K : 
\forall gK \in G/K
\vq(\tau(g)t) = (\sigma(g),\rho(g))\cdot \vq(t)
\},
\]
one has for the cyclic case 
\[
X^G \cong 
\{ \vq \from [0,1] \to (E^n)^K :
\vq(1) = (\sigma(r),\rho(r))\cdot \vq(0) 
\} 
\]
and for the brake and dihedral case
\[
X^G \cong 
\{ \vq \from [0,1] \to (E^n)^K :
\vq(0) = (\sigma(h_0),\rho(h_0))\cdot \vq(0) \wedge 
\vq(1) = (\sigma(h_1),\rho(h_1))\cdot \vq(1)
\}.
\]
Now, $(E^n)^K$ is just an euclidean subspace of $E^n$, on which $G/K$ acts. 
Consider therefore the fact that 
\[
H^1([0,1], (E^n)^K) \cong (E^n)^K\oplus  
H^1_0([0,1],(E^n)^K) \oplus
(E^n)^K 
\]
by sending $\vq(t)$ to $(\vq(0),\vq(t) - ( \vq(0) + t (\vq(1)-\vq(0)),\vq(1))$. 

So the group of symmetries acts either as $\ker \tau$ (and hence it 
gives constraints to the points in 
the configuration space for all $t$),
or it acts just on a finite dimensional subspace of the 
Hilbert space $X^G$, by adding constraints  on the boundary of the fundamental domain $[0,1]$.

\section{Local minimizers of symmetric Lagrangean functionals}
\textbf{Approximations and projections}

As we have seen before, we need to define finitely-dimensional 
linear subspaces of the Hilbert space $H^1_0([0,1],(E^n)^K)$.
It is natural to consider Fourier sine polynomials of type
\[
\gamma(t) = \sum_{k=1}^s \va_k \sin(k\pi t)
\]
for $s\geq 1$ and $\va_k \in (E^n)^K$. 
In other words, we can consider embedding of the finite-dimensional space 
$(E^{n(s+2)})$ in $X$ defining
for each $\va$ in $E^n\oplus (E^n)^s \oplus E^n$ 
\begin{equation}\label{eq:embedding}
\Psi(\va_0,\va_1,\ldots, \va_s, \va_{s+1}) 
= 
\va_0 + t( \va_{s+1}-\va_0 ) + \sum_{k=1}^s \va_k \sin(k\pi t)
\end{equation}
Now, the action of $G$ is diagonal as $\sigma \times \rho$ on all
the copies $E^n$ with $0\leq k \leq s+1$. 
The symmetry constraints on the boundary will be 
$r \va_0 = \vq_1$ if of cyclic type,
otherwise 
$h_j \va_j =$ for $j=0,1$. 
Both can be written as a single involution $h$ defined on $E^{n}\oplus E^n$ as 
\[
\begin{aligned}
h(\va_0,\va_{s+1}) = ( r^{-1} \va_{s+1}, r \va_{0} ) &\quad  \text{if cyclic type} \\
h(\va_0,\va_{s+1}) = ( h_0 \va_{0}, h_1 \va_{s+1} ) & \quad \text{if brake or dihedral type} \\
\end{aligned}
\]
The configurations $\va_0$ and $\va_{s+1}$ are the configurations at time $0$ and 
$1$ respectively, and the configurations $\va_k$ with $1\leq k \leq s$ are 
Fourier coefficients of the term in $H^1_0$. All of them need to be 
made symmetric with respect to the group $K$. 

\begin{remark}
Such an approximation raises some questions and issues about 
visibility of critical points, i.e. whether  the sequence of finite-dimensional
subspaces of the Hilbert space $X$ approximate all critical points 
of $\action$. A positive answer can be found in 
\cite{MR2678672}, 
provided some natural non-degeneracy assumptions are made.
Other possible approximations compatible with symmetries are PL paths.
\end{remark}

\begin{lemma}
If $\vq(t) = \Psi(\va)$ for $\va \in E^{n(s+1)}$, then 
the kinetic part of the Lagrangian is 
\[
\int_0^T 
\dfrac{1}{2} \mnorm{\dot \vq}^2 \de{t} = 
\dfrac{|G/K|}{2} 
\int_0^1 
\mnorm{\dot \vq}^2 \de{t}
\]
where
\[
\int_0^1 \mnorm{ \dot \vq}^2 \de{t} = 
\mnorm{ \va_{s+1} - \va_0}^2  + 
\sum_{k=1}^s
\frac{k^2\pi^2}{2} \mnorm{\va_k}^2~.
\]
\end{lemma}
\begin{proof}
\[
\begin{aligned}
\dfrac{d \vq}{dt} & = \va_{s+1} - \va_0 + \sum_{k=1}^s k\pi \va_k \cos( k\pi t )\\
\implies 
\mnorm{\dot \vq}^2 & = 
\mnorm{ \va_{s+1} - \va_0}^2 + 2 
\sum_{k=1}^s \mscalar{\va_{s+1} - \va_0}{\va_k } k \pi \cos (k\pi t) 
+ \\
& +\mnorm{ 
\sum_{k=1}^s  k \pi \va_k \cos(k\pi t)
}^2 
\\
\implies
\int_0^1 
\mnorm{\dot \vq}^2 \de{t}  & = 
\mnorm{ \va_{s+1} - \va_0}^2 
+
\sum_{j=1}^s
\sum_{k=1}^s 
\mscalar{ j \pi \va_j }{ k \pi \va_k} \int_0^1 \cos(j\pi t)\cos(k \pi t) \de{t} \\
& =
\mnorm{ \va_{s+1} - \va_0}^2  + 
\sum_{k=1}^s
\frac{k^2\pi^2}{2} \mnorm{\va_k}^2 
\end{aligned}
\]
since for integers $j,k$
\[
\begin{aligned}
k \geq 1 & \implies 
\int_0^1 \cos(k\pi t) \de{t} = 0 , \quad 
\int_0^1 \cos^2(k\pi t) \de{t}  = \frac{1}{2} \\
j\neq k & \implies 
\int_0^1 \cos(j\pi t) \cos(k\pi t) \de{t}   = 0.
\end{aligned}
\]
\end{proof}

Now, 
the kinetic part is a plain quadratic form on $\va$, with explicit coefficients. 
About the potential part, it involves some integrals which 
cannot be computed symbolically, and hence another level of approximation must occur. 
Applying a composite trapezoidal rule we can 
subdivide the interval $[0,1]$  in $\nop$ equal intervals 
of length $\frac{1}{\nop} $, and define an approximate potential as
\begin{equation}
\label{eq:trapezoidal}
\int_0^1 U(\vq(t)) \de{t} \approx
\mathcal U( \vq ) = 
\dfrac{1}{\nop}
\left(
\dfrac{U(\vq(0))}{2}
+
\sum_{j=1}^{\nop} 
U(\vq(\frac{j}{\nop}))
+ 
\dfrac{U(\vq(1))}{2}
\right)
\end{equation}
which implies that if $\vq = \Psi(\va)$ the following lemma holds,
once we write in short $\mathcal U(\va) = \mathcal U(\Psi(\va))$. 

\begin{lemma}
If $\va \in E^{n(s+1)}$ and $\nop\geq 1$, then
\[
\nop \mathcal U (\va) = 
\dfrac{U(\va_0)}{2}
+
\sum_{j=1}^{\nop} 
U\left ( (1-\frac{j}{\nop})\va_0 + \frac{j}{\nop}\va_{s+1} + 
\sum_{k=1}^s \va_k \sin(\frac{jk\pi}{\nop} )
\right)
+ 
\dfrac{U(\va_{s+1})}{2}
\]
\end{lemma}
\begin{proof}
If $\vq = \Psi(\va)$, then 
for $j=0,\ldots, \nop$
one has $t_j = \frac{j}{\nop} \in [0,1]$ and
\[
\begin{aligned}
\vq(t_j) = 
\va_0 + t_j (\va_{s+1} - \va_0) + 
\sum_{k=1}^s \va_k \sin(k\pi t_j)
\end{aligned}
\]
hence
the claim by \eqref{eq:trapezoidal}.
\end{proof}

Now, in order to apply numerical algorithms to find local minima or critical
points of $\mathcal U$ constrained to the symmetric paths,
we might need to compute the derivatives of $\mathcal U$ 
with respect to the components of $\va$ (which are $n \times d \times (s+2)$
real variables).  
This can be done (keeping track of symmetries) by formally 
deriving the expressions of the functional.

\section{Existence theorems}
\label{sec:existence}

Due to collisions and lack of compactness, 
theorems stating the existence of periodic orbits  for 
the $n$-body problem as critical points of 
the Lagrangean functional were not easy to prove. 
A significant breakthrough in the
variational approach, is due to the introduction of the \emph{strong force} condition by
 Gordon in 1975
\cite{gordon}.
With different methods (analytical and topological) 
several solutions were proved to exist for $n$-body type 
problems, provided the (asymptitical) homogeneity of the potential
in \eqref{eq:potential} 
is $\alpha \geq 2$. 
Such approach has been extended and generalized in several directions:
Ambrosetti and Coti-Zelati 1987 \cite{MR89e:58023}, 
Bahri-Rabinowitz 1991 \cite{MR1145561},
Bessi and Coti-Zelati 1991 \cite{bessi},
Fadell and Husseini 1992 \cite{MR1168305}, %
Majer and Terracini 1993 \cite{MR1240581}. %
See Ambrosetti and Coti-Zelati monograph 
\cite{amco} for further details. 
A new wave of results for symmetric periodic orbits 
followed Chenciner and Montgomery remarkable figure-eight orbit \cite{monchen},
 and its generalizations by the author and Terracini \cite{FT2003}. 
In these results the strong-force assumption is not necessary,
and 
after 
Marchal's averaging trick \cite{MR1956531}
variational methods for symmetric orbits have been 
extensively studied in several articles that cannot be fully cited 
here.

Let $\XX = \{ \vq \in E^n : \sum_{j=1}^n m_j \vq_j = \zero \}$
be the space of all positions with center of mass in zero. 
Then $\XX$ is invariant with respect to the action of $G$ 
on $E^n$. The following lemma gives a complete characterization 
of $G$-actions for which the Lagrangian action functional 
is coercive on 
$X^G$. 
\begin{lemma}[Proposition (4.1) of \cite{FT2003}]
\label{lem:coercive}
The restriction $\action^G\from X^G \to \RR$ is coercive if and only if 
$\mathcal{X}^G = \zero$.  
\end{lemma}

Existence of non-colliding (local) minimizers in $\action^G$, 
under suitable conditions on the action 
(provided not all trajectories in $\action^G$ are bound to collisions)
can be a consequence of two types of results. 
Either a local variation (using avaraging methods and blow-up asymptotic 
analysis on parabolic trajectories) is shown to decrease 
the action on all colliding trajectories, 
or  more precise action level estimates are performed, together with
some more topological local analysis.  
We quote here some of the relevant theorems for symmetric orbits,
leaving the details to the cited references.  

\begin{theo}[Theorem (10.10) of \cite{FT2003}]
If $\ker \tau\subset G$ 
and (if they exist) the $\TT$-isotropy subgroups generated by $h_0$ and $h_1$
either have the rotating circle property or act trivially on the index set $\vv$,
then any local minimizer of $\action^G$ in $X^G$ does not have collisions.
\end{theo}

\begin{theo}[Theorem (6.11) of \cite{MR2439573}]
\label{theo:averaging} %
Assume that $U^{\ker \tau}$ has the form 
\[
U(\vx) = \sum_{j=1}^k \dfrac{\Gamma_j}{\dist(\vx, V_j)^\alpha}
\] 
for each $\vx \in \XX$, 
for $k$ positive constants $\Gamma_j>0$ 
and $k$ linear subspaces of $\XX$ of codimension $\codim V_j \geq 2$ 
in $\XX^{\ker \tau}$. 
\end{theo}

\begin{theo}[Theorem (4.1) of \cite{MR2819162}]
If $\ker\tau$ is one of the platonic finite subgroups of $SO(3)$,
then in any homotopy class of loops of $\ker\tau$-equivariant configurations
in an explicit list (table 1 of \cite{MR2819162}), 
there is a non-colliding local minimizer of $\action^{\ker \tau}$. 
\end{theo}

So, more or less the problem is to approximate such local mimizers,
and to explore all possible symmetry groups yielding non-colliding 
solutions. 
After Chenciner-Montgomery figure-eight remarkable solution \cite{monchen}, 
several authors developed and published numerical schemes or computer-assisted methods 
relative to the problem. See for example 
\cite{MR1916506,simo2001,kapela2003}. %
By exploiting the computer algebra symbolic power of GAP
\cite{GAP4}, and the Geometric visualization engine \texttt{geomview}, 
we developed a multi-faceted approch, with the package ``symorb'', 
published on \texttt{github} at \url{https://github.com/dlfer/symorb}. 

\section{Remarks}

\begin{remark}
Computing techniques have played and are playing a role also in another 
sub-problem of the $n$-body problem: finding and studying \emph{central configurations}. 
The problem is simple: study all critical points (in the sense of Morse-Bott,
because of the Euclidean symmetry group) of the potential function  $U$ 
\eqref{eq:potential} restricted 
to the \emph{inertia ellipsoid} 
$S = \{ \vq \in \XX : \mnorm{vq}=1 \}$. 
For $n=3$, the solution can be traced back to Lagrange ($d\geq 2$) and Euler ($d=1$). 
For any $n$ and $d=1$ the problem was solved by Moulton at the beginning
of '900. 
For $n\geq 4$ and $d\geq 2$, 
it is not difficult to use a computer to find some solutions, but it is 
remarkably more difficult to use a computer algebra system to prove results,
such as giving estimates of the number (finite or not) of classes of solutions. 
In
\cite{HaM2207019}, Moeckel and Hampton proved that the number 
of equivalence classes of central configurations for $n=4$ and $d=2$ is finite
(in the interval [32,8472]), with some careful computer computations. 
Later Albouy and Kaloshin \cite{MR2925390} extended the finiteness result to $n=5$,
and simplified the proof for $n=4$ (but still using a Computer Algebra System
to finish the proof).
In this line of research, we have been building some (pre-computational) tools 
aimed at transforming the central configurations problem into
 a fixed-point problem \cite{MR2372989,Fer2015,MutDiff2017}.
See 
also computer assisted proofs by 
Moczurad and 
Zgliczynski in 
\cite{moczuradCentralConfigurationsPlanar2019}
and \cite{moczuradCentralConfigurationsSpatial2020}.
\end{remark}


\begin{thebibliography}{10}

\bibitem{MR2925390}
Alain Albouy and Vadim Kaloshin.
\newblock Finiteness of central configurations of five bodies in the plane.
\newblock {\em Ann. of Math. (2)}, 176(1):535--588, 2012.

\bibitem{MR89e:58023}
Antonio Ambrosetti and Vittorio Coti~Zelati.
\newblock Critical points with lack of compactness and singular dynamical
  systems.
\newblock {\em Ann. Mat. Pura Appl. (4)}, 149:237--259, 1987.

\bibitem{amco}
Antonio Ambrosetti and Vittorio Coti~Zelati.
\newblock {\em Periodic solutions of singular {L}agrangian systems}.
\newblock Birkh\"auser Boston Inc., Boston, MA, 1993.

\bibitem{MR1145561}
A.~Bahri and P.~H. Rabinowitz.
\newblock Periodic solutions of {H}amiltonian systems of {$3$}-body type.
\newblock {\em Ann. Inst. H. Poincaré Anal. Non Linéaire}, 8(6):561--649,
  1991.

\bibitem{MR2439573}
Vivina Barutello, Davide~L. Ferrario, and Susanna Terracini.
\newblock On the singularities of generalized solutions to {$n$}-body-type
  problems.
\newblock {\em Int. Math. Res. Not. IMRN}, pages Art. ID rnn 069, 78, 2008.

\bibitem{bessi}
Ugo Bessi and Vittorio Coti~Zelati.
\newblock Symmetries and noncollision closed orbits for planar ${N}$-body-type
  problems.
\newblock {\em Nonlinear Anal.}, 16(6):587--598, 1991.

\bibitem{monchen}
Alain Chenciner and Richard Montgomery.
\newblock A remarkable periodic solution of the three-body problem in the case
  of equal masses.
\newblock {\em Ann. of Math. (2)}, 152(3):881--901, 2000.

\bibitem{Devaney81}
Robert~L. Devaney.
\newblock Singularities in classical mechanical systems.
\newblock In {\em Ergodic theory and dynamical systems, I (College Park, Md.,
  1979--80)}, volume~10 of {\em Progr. Math.}, pages 211--333. Birkh\"auser
  Boston, Mass., 1981.

\bibitem{MR1168305}
E.~Fadell and S.~Husseini.
\newblock Infinite cup length in free loop spaces with an application to a
  problem of the {$N$}-body type.
\newblock {\em Ann. Inst. H. Poincar\'e Anal. Non Lin\'eaire}, 9(3):305--319,
  1992.

\bibitem{Fer2015}
D.~L. Ferrario.
\newblock Fixed point indices of central configurations.
\newblock {\em J. Fixed Point Theory Appl.}, 17(1):239--251, 2015.

\bibitem{MutDiff2017}
D.~L. Ferrario.
\newblock Central configurations and mutual differences.
\newblock {\em SIGMA}, 13(021):11, 2017.

\bibitem{MR2372989}
Davide~L. Ferrario.
\newblock Planar central configurations as fixed points.
\newblock {\em J. Fixed Point Theory Appl.}, 2(2):277--291, 2007.

\bibitem{FT2003}
Davide~L. Ferrario and Susanna Terracini.
\newblock On the existence of collisionless equivariant minimizers for the
  classical {$n$}-body problem.
\newblock {\em Invent. Math.}, 155(2):305--362, 2004.

\bibitem{MR2819162}
G.~Fusco, G.~F. Gronchi, and P.~Negrini.
\newblock Platonic polyhedra, topological constraints and periodic solutions of
  the classical {$N$}-body problem.
\newblock {\em Invent. Math.}, 185(2):283--332, 2011.

\bibitem{GAP4}
The GAP~Group.
\newblock {\em {GAP -- Groups, Algorithms, and Programming, Version 4.8.7}},
  2017.

\bibitem{gordon}
William~B. Gordon.
\newblock Conservative dynamical systems involving strong forces.
\newblock {\em Trans. Amer. Math. Soc.}, 204:113--135, 1975.

\bibitem{HaM2207019}
Marshall Hampton and Richard Moeckel.
\newblock Finiteness of relative equilibria of the four-body problem.
\newblock {\em Invent. Math.}, 163(2):289--312, 2006.

\bibitem{kapela2003}
Tomasz Kapela and Piotr Zgliczy\'nski.
\newblock The existence of simple choreographies for the {$N$}-body problem---a
  computer-assisted proof.
\newblock {\em Nonlinearity}, 16(6):1899--1918, 2003.

\bibitem{MR1240581}
Pietro Majer and Susanna Terracini.
\newblock Periodic solutions to some problems of {$n$}-body type.
\newblock {\em Arch. Rational Mech. Anal.}, 124(4):381--404, 1993.

\bibitem{MR1956531}
C.~Marchal.
\newblock How the method of minimization of action avoids singularities.
\newblock {\em Celestial Mech. Dynam. Astronom.}, 83(1-4):325--353, 2002.
\newblock Modern celestial mechanics: from theory to applications (Rome, 2001).

\bibitem{MR0495348}
J.~N. Mather and R.~McGehee.
\newblock Solutions of the collinear four body problem which become unbounded
  in finite time.
\newblock In {\em Dynamical systems, theory and applications ({R}encontres,
  {B}attelle {R}es. {I}nst., {S}eattle, {W}ash., 1974)}, pages 573--597.
  Lecture Notes in Phys., Vol. 38. Springer, Berlin, 1975.

\bibitem{moczuradCentralConfigurationsPlanar2019}
Ma{\l}gorzata Moczurad and Piotr Zgliczy{\'n}ski.
\newblock {C}entral {C}onfigurations in {P}lanar {N}-{B}ody {P}roblem with
  {E}qual {M}asses for $n=5,6,7$.
\newblock {\em Celestial Mechanics and Dynamical Astronomy}, 131(10):46,
  September 2019.

\bibitem{moczuradCentralConfigurationsSpatial2020}
Ma{\l}gorzata Moczurad and Piotr Zgliczy{\'n}ski.
\newblock {C}entral {C}onfigurations in the {S}patial {N}-{B}ody {P}roblem for
  $n=5,6$ with {E}qual {M}asses.
\newblock {\em Celestial Mechanics and Dynamical Astronomy}, 132(11):56,
  December 2020.

\bibitem{MR1916506}
Michael Nauenberg.
\newblock Periodic orbits for three particles with finite angular momentum.
\newblock {\em Phys. Lett. A}, 292(1-2):93--99, 2001.

\bibitem{MR2678672}
F.~Sani and M.~Villarini.
\newblock Detectability of critical points of smooth functionals from their
  finite-dimensional approximations.
\newblock {\em Nonlinear Anal.}, 73(9):3140--3150, 2010.

\bibitem{simo2001}
Carles Sim\'o.
\newblock New families of solutions in {$N$}-body problems.
\newblock In {\em European {C}ongress of {M}athematics, {V}ol. {I}
  ({B}arcelona, 2000)}, volume 201 of {\em Progr. Math.}, pages 101--115.
  Birkh\"auser, Basel, 2001.

\bibitem{MR1166640}
Zhihong Xia.
\newblock The existence of noncollision singularities in {N}ewtonian systems.
\newblock {\em Ann. of Math. (2)}, 135(3):411--468, 1992.

\end{thebibliography}
\end{document}